\documentclass[11pt, letterpaper]{amsart}

\usepackage{amsmath}
\usepackage{amssymb, mathrsfs, mathtools}
\usepackage{amsfonts, mathrsfs}
\usepackage{amsthm}
\usepackage{relsize}
\usepackage{enumerate}
\usepackage{verbatim}
\usepackage{amscd}
\usepackage{geometry}
\usepackage{graphicx}
\usepackage{appendix}
\usepackage{tikz}
\usepackage{hyperref}
\usepackage{kotex}
\usepackage{bbm}

\newtheorem{innercustomgeneric}{\customgenericname}
\providecommand{\customgenericname}{}

\newcommand{\newcustomtheorem}[2]{\newenvironment{#1}[1]
  {\renewcommand\customgenericname{#2}
   \renewcommand\theinnercustomgeneric{##1}\innercustomgeneric}{\endinnercustomgeneric}}

\newcustomtheorem{customthm}{Theorem}

\newcommand{\newcustomlemma}[2]{\newenvironment{#1}[1]
  {\renewcommand\customgenericname{#2}
   \renewcommand\theinnercustomgeneric{##1} \innercustomgeneric}{\endinnercustomgeneric}}

\newcustomlemma{customlemma}{Lemma}

\newcustomlemma{customproposition}{Proposition}

\theoremstyle{plain}
\newtheorem{theorem}{Theorem}[section]

{}
\newtheorem{lemma}[theorem]{Lemma}

\newtheorem{corollary}[theorem]{Corollary}
\theoremstyle{definition}

\theoremstyle{remark}
\newtheorem{remark}{Remark}

\numberwithin{equation}{section}

\newcommand{\bZ}{\mathbb{Z}}
\newcommand{\bR}{\mathbb{R}}

\newcommand{\bA}{\mathbb{A}}
\newcommand{\bB}{\mathbb{B}}

\newcommand{\re}{\mathrm{e}}

\newcommand{\rA}{\mathrm{A}}
\newcommand{\rF}{\mathrm{F}}
\newcommand{\rR}{\mathrm{R}}
\newcommand{\rM}{\mathrm{M}}

\newcommand{\cF}{\mathcal{F}}

\newcommand{\supp}{\mathrm{supp}}

\begin{document}

\title[Multilinear maximal averages]{Improved curvature conditions on $L^2\times\cdots\times L^2 \to L^{2/m}$ bounds for multilinear maximal averages}

\author{Chu-hee Cho}
\address{C. Cho, Research Institute of Mathematics, Seoul National University, 08826 Gwanak-ro 1, Seoul, Republic of Korea}
\email{akilus@snu.ac.kr}

\author{Jin Bong Lee}
\address{J. B. Lee, Research Institute of Mathematics, Seoul National University, 08826 Gwanak-ro 1, Seoul, Republic of Korea}
\email{jinblee@snu.ac.kr}

\author{Kalachand Shuin}
\address{K. Shuin, Department of Mathematical Sciences, Seoul National University, 08826 Gwanak-ro 1, Seoul, Republic of Korea}
\email{kcshuin21@snu.ac.kr}

\subjclass[2020]{42B25, 47H60}

\keywords{Multilinear maximal averages, Curvature conditions}

\maketitle
\begin{abstract}
	In this article, we focus on $L^{2}(\mathbb{R}^d)\times\cdots\times L^{2}(\mathbb{R}^d)\rightarrow L^{2/m}(\mathbb{R}^d)$ estimates for  multilinear maximal averages over non-degenerate hypersurfaces. 
	Our findings is new for $m$-linear averages with $m\geq3$, and represent a reproof of the recent result of T. Borges, B. Foster, and Y. Ou on the curvature conditions of the hypersurfaces required in establishing  $L^{2}(\mathbb{R}^d)\times L^{2}(\mathbb{R}^d)\rightarrow L^{1}(\mathbb{R}^d)$ estimates of bilinear maximal functions.
\end{abstract}

\section{Introduction}

Let $\sigma$ be the normalized surface measure on $\Sigma$ supported in a unit ball $\bB^{md}(0,1) \subset \bR^{md}$.
For $\mathrm{F} = (f_1, \dots, f_m)$ with Schwartz functions $f_1, \dots, f_m \in \mathscr{S}(\bR^d)$, define a multilinear averaging operator $\mathrm{A}_\sigma$ by 
\begin{align}\label{defn_ma}
	\mathrm{A}_\sigma(\mathrm{F})(x,t) \coloneqq \int_{\Sigma} \prod_{i=1}^m f_i(x-ty_i)~\mathrm{d}\sigma(y),\end{align}
and a maximal average by
\begin{align}\label{defn_mma}
	\mathrm{M}_\sigma(\rF)(x) \coloneqq \sup_{t>0}|\mathrm{A}_\sigma(\rF)(x,t)|.
\end{align}

One classical example of multilinear maximal function is the bilinear spherical maximal function, which is defined by 
$$
	\mathcal{M}_{full}(\mathrm{F})(x):=\sup_{t>0}\Big|\int_{\mathbb{S}^{2d-1}}f_1(x-ty)f_2(x-tz)~d\sigma(y,z)\Big|,
$$
where $\sigma$ is the normalized surface measure on $\mathbb{S}^{2d-1}$, $d\geq1$.
This operator was initially introduced by Barrionuevo, Grafakos, He, Honz\'ik, and Oliveira \cite{BGHHO_2018}, and then Heo, Hong, and Yang \cite{HHY_2020} improved upon the previous results. 
After, Jeong and Lee \cite{JeongLee_2020} have successfully demonstrated the complete (except few border line cases) $L^{p_1}\times L^{p_2}\rightarrow L^{p}$ boundedness of this operator using a clever idea of slicing argument in dimensions $d\geq2$.

Boundedness of the bilinear spherical maximal function in dimension $d=1$ was later investigated by Christ, Zhou \cite{ChZh_2022} and Dosidis, Ramos \cite{DoRa_2022} independently. 
Recently, Bhojak, Choudhary, Shrivastava, and the third author of this article \cite{BCSS_2023} have established end point estimates of bilinear spherical maximal function in dimensions $d=1,2$. 
Lee and the third author of this article \cite{LeeShuin_2023} extended the slicing technique to deal with boundedness of bilinear maximal functions defined on degenerate hypersurfaces, like $\Sigma_{a_1,a_2}:=\{(y,z)\in \mathbb{R}^{2d}:\Phi(y,z)=|y|^{a_1}+|z|^{a_2}-1=0\}$, for $a_1,a_2\in [1,\infty)$. 
Despite of the vanishing curvature conditions of $\Sigma_{a_1,a_2}$, the authors of \cite{LeeShuin_2023} were able to apply the slicing argument to study boundedness of the bilinear maximal function defined on $\Sigma_{a_1,a_2}$ and proved sharp  $L^{p_1}\times L^{p_2}\rightarrow L^{p}$ estimates except for border line cases. 
In the slicing argument of \cite{LeeShuin_2023}, the non-vanishing gradient $(|\nabla \Phi|\neq 0)$ of the hypersurface representing function $\Phi$ plays a crucial role.

Therefore, a natural question arises, when we only have information about  the curvature conditions of hypersurfaces but not the equation, what can we infer about the boundedness of bilinear and multilinear maximal functions associated with the hypersurfaces? 
This problem was first addressed by Grafakos, He, and Honz\'ik \cite{GHH_2021} and later Chen, Grafakos, He, Honz\'ik, Slav\'{i}kov\'{a} \cite{CGHHS_2022} considered as bilinear analogues of \cite{RdF_1986}.
For $m=2$, Chen $et.$ $al.$ \cite{CGHHS_2022} proved the following result.
\begin{theorem}[\cite{CGHHS_2022}, Theorem 2]
	Let $\sigma$ be the surface measure of a compact and smooth surface  $\Sigma$ without boundary such that $k$ of its $2d-1$ principal curvatures are non-zero. 
	Then $\mathrm{M}_\sigma$ maps $L^{2}(\mathbb{R}^{d})\times L^{2}(\mathbb{R}^{d})\rightarrow L^{1}(\mathbb{R}^{d})$ when  $k>d+2$.
\end{theorem}
Very recently, Borges, Foster, and Ou \cite{BoFoOu_2023} have improved the curvature condition to $k\geq d+2$.
Their results are based on studies of Sobolev smoothing estimates for various bilinear maximal operators given by Fourier multipliers, which contain multi-scale maximal functions, and maximal functions with fractal dilation sets. 

To state our main results, we begin with mutlilinear local maximal functions.
\begin{align}
	\rM_\sigma^{loc}(\mathrm{F})(x) \coloneqq \sup_{1<t<2} \Big| \int \prod_{i=1}^m f_i(x- t y_i)~\mathrm{d}\sigma(y)\Big|.
\end{align}
We first establish multilinear local maximal estimates.
\begin{theorem}\label{thm_loc}
	Let $s>\frac{(m-1)d}{2} + \frac{1}{2}$, $m\geq2$, and a measure $\sigma$ be supported in $\bB^{md}(0,1)$ and satisfy
	\[
		|\widehat{\mathrm{d}\sigma}(\xi)|\lesssim (1+|\xi|)^{-s}.
	\]
	Then we have for $2/m\leq p\leq2$,
	\[
		\| \rM_\sigma^{loc}(\rF) \|_{L^{p}(\bB^d(0,1))} \leq C \prod_{i=1}^m \|f_i\|_{L^2(\bR^d)}.
	\]
	For $p=2/m$, we have
	\[
		\| \rM_\sigma^{loc}(\rF) \|_{L^{2/m}(\bR^d)} \leq C \prod_{i=1}^m \|f_i\|_{L^2(\bR^d)}.
	\]
\end{theorem}

By making use of Theorem~\ref{thm_loc}, we obtain estimates for global maximal functions \eqref{defn_mma}.
\begin{theorem}\label{thm_main}
	Let $m$, $d$, $\sigma$ be given as in Theorem~\ref{thm_loc}.
	Then we have
	$$
		\| \rM_\sigma(\rF) \|_{L^{2/m}(\bR^d)} \leq C \prod_{i=1}^m \|f_i\|_{L^2(\bR^d)}.
	$$
\end{theorem}

\begin{remark}\label{rem_counterex}
         In order to discuss optimal range of the above estimates, we consider two operators. First, let $\mathrm{M}_{\mathbb{S}^2}$ denote the spherical maximal function on $\bR^3$ whose Fourier symbol has decay $1$.
	Then, $m$-product of $\mathrm{M}_{\mathbb{S}^2}$ obeys that
	$$
		\Big\| \prod_{i=1}^m\mathrm{M}_{\mathbb{S}^2}(f_i)\Big\|_{L^p(\bR^3)} \leq C \prod_{i=1}^m \|f_i\|_{L^{p_i}(\bR^3)},
	$$
	whenever $p>3/(2m)$, $p_i>3/2$ with $1/p = \sum_{i=1}^m 1/p_i$.
	This certainly implies $L^2(\bR^3)\times\cdots\times L^2(\bR^3) \to L^{2/m}(\bR^3)$ estimates.
	On the other hand, for the circular maximal function $\mathrm{M}_{\mathbb{S}^1}$, its Fourier decay of the circular measure is $1/2$ and its $m$-product does not satisfy                   $L^2(\bR^2)\times\cdots\times L^2(\bR^2) \to L^{2/m}(\bR^2)$ estimates
	since $\mathrm{M}_{\mathbb{S}^1}$ is bounded on $L^p(\bR^2)$ if and only if $p>2$. In particular, in bilinear case of Theorems~\ref{thm_loc} and \ref{thm_main}, the condition $s>\frac{(m-1)d}{2} + \frac{1}{2}$ becomes $s>\frac{d+1}{2}$.
	However, it should be noted that $s>\frac{d+1}{2}$ is not sharp for Theorems~\ref{thm_loc} and \ref{thm_main}, and in terms of $k$ nonvanishing principal curvatures, $s>\frac{d+1}{2}$ is equivalent to $k>d+1$ same as \cite{BoFoOu_2023}.
	In fact, from the results for two maximal operators, one may conjecture for the case of $m=2$ that optimal range of $s$ for Theorems~\ref{thm_loc} and \ref{thm_main} is $s>1/2$.
	\end{remark}
By multilinear real interpolation, we obtain $L^p(\bR^d)\times\cdots\times L^p(\bR^d)\to L^{p/m}(\bR^d)$ estimates of $\mathrm{M}^{loc}_{\sigma}$ for some $p<2$.
	\begin{corollary}\label{Prop1}
		Let $m\geq2$ and $\Sigma$ be a compact and smooth hypersurface with non-vanishing principal curvatures $k>(m-1)d+1$. Then, the multilinear local maximal function $\mathrm{M}^{loc}_{\sigma}$ maps $L^{p}(\mathbb{R}^{d})\times\cdots\times L^{p}(\mathbb{R}^{d})$ to $L^{p/m}(\mathbb{R}^{d})$ for $p>\frac{k-(m-1)d+1}{k-(m-1)d}$. 
	\end{corollary}
	By Corollary~\ref{Prop1}, it is shown for $m=2$ that $M_\sigma^{loc}$ satisfies $L^p(\bR^d)\times L^p(\bR^d) \to L^{p/2}(\bR^d)$ estimates for $p>\frac{k-d+1}{k-d}$.
	However, the range of $p$ is not sharp as we mentioned in Remark~\ref{rem_counterex}.
	It could be conjectured that $p>(k+1)/k$ is sharp, which is obtained from validity of Theorems~\ref{thm_loc} and \ref{thm_main} for $s>1/2$.

\section*{Notations}
Let $\bB^d(x,R)$ denote a $d$-dimensional ball of radius $R$ centered at $x$.
Let $\bA^d(\lambda)$ be a $d$-dimensional annulus given by $\{x\in\bR^d : 2^{-1}\lambda <|x| < 2\lambda\}$.

\section{Preliminaries}

\subsection{$\alpha$-dimensional measures and weighted estimates}

Contents of this subsection are mostly given by Ko, Lee, and Oh \cite{KoLeeOh_2022}.
Let $\mu$ be a positive Borel measure on $\bR^d$.
Then for $\alpha\in(0,d]$, $\mu$ is said to be $\alpha$-dimensional if there exists a $C>0$ such that
\begin{align}
	\mu(\bB^d(x_0,R)) \leq C R^\alpha,
\end{align}
where $\bB^d(x_0,R)$ denotes a $d$-dimensional ball centered at $x_0$ with radius $R$.
For an $\alpha$-dimensional measure $\mu$, one can define the following quantity:
\begin{align}
	\langle \mu\rangle_\alpha  = \sup_{x_0 \in \bR^d, R>0} R^{-\alpha} \mu(\bB^d(x_0, R)).
\end{align}

It is known in the literature that one can rewrite $\mathrm{M}_\sigma f$ by making use of a measurable function $\bold{t}(\cdot)$:
\begin{align}
	\mathrm{M}_\sigma f(x) = \int f(x- \bold{t}(x)y)~\mathrm{d}\sigma(y).
\end{align}
We define a $d$-dimensional measure $\mu$ by
\begin{align}
	\int_{\bR^{d+1}} F(x,t)~\mathrm{d}\mu(x,t) = \int_{\bB^d(0,1)} F(x, \bold{t}(x))~\mathrm{d}x
\end{align}
for any function $ F\in C_c(\bR^{d+1})$.
Then, $L^p(\bB^d(0,1), \mathrm{d}x)$ norm of $\mathrm{M}_\sigma$ is equivalent to
\begin{align}\label{a_dim_est}
	\Big( \int | \mathrm{A}_\sigma f(x,t) |^p ~\mathrm{d}\mu(x,t) \Big)^{1/p}
\end{align}
That is, $\mathrm{d}\mu$ is a $d$-dimensional measure on $\bR^d \times (0, \infty)$.

We restrict our consideration of $\alpha$-dimensional measures on $\mu$ given by $\omega~\mathrm{d}x\mathrm{d}t$.
Let $\Omega^\alpha$ be a collection of non-negative measurable functions $\omega$ on $\bR^{d+1}$ such that $\omega~\mathrm{d}x\mathrm{d}t$ is $\alpha$-dimensional.
We use a notation $[\omega]_\alpha = \langle \omega~\mathrm{d}x\mathrm{d}t\rangle_\alpha$.
It is known in \cite{KoLeeOh_2022} that \eqref{a_dim_est} is reduced to weighted norms associated with $\omega \in \Omega^d$.
The following lemma was already given in \cite[Lemma~2.19]{KoLeeOh_2022} for dimension $d=4$. One can easily extend it to general dimension $d$ by modifying the proof of the case of $d=4$, so we omit it.
\begin{lemma}\label{lem_a_meas}
	Let $p\in(0, \infty]$, $\alpha\in(0,d+1]$ and $\omega \in \Omega^\alpha$.
	For a function $F \in L^p(\bR^{d+1})$ whose Fourier support is a subset of $\bB^{d+1}(0, \lambda)$, 
	we have
	\begin{align}
		\| F\|_{L^p(\bR^{d+1}, \omega)} \leq C[\omega]_\alpha^{\frac1p} \lambda^{\frac{d+1 - \alpha}{p}} \|F\|_{L^p(\bR^{d+1}, \mathrm{d}x\mathrm{d}t)}.
	\end{align}
\end{lemma}

\subsection{$L^p$ improving estimates for multilinear operators with the limited decay condition}

We recall simple estimates for multilinear operators whose symbols satisfy the following decay conditions:
\[
	| \mathfrak{m}(\xi)| \lesssim (1+ |\xi|)^{-s}
\]
for some given $s>0$.
We define the Littlewood-Paley decomposition for further uses.
Let $\Psi$ be a Schwartz function on $\bR^{md}$ such that $\widehat{\Psi}\equiv1$ for $1<|\xi|<2$ and vanishes outside of $\{\xi \in \bR^{md} : 1/2<|\xi|<4\}$.
Define $\Psi_j(\cdot) = 2^{jmd}\Psi(2^j\cdot)$ and $\Phi(\cdot) = 1 - \sum_{j\geq1}\Psi_j$.
If we use the decomposition on $\bR^d$, we say $P_j$ to denote the $j$-th Littlewood-Paley projection.

Then, we have the following lemma:
\begin{lemma}\label{lem_decay}
	Let $m\geq2$ and $T_{\mathfrak{m}}(\rF)$ be a multilinear operator given by
	$$
		T_{\mathfrak{m}}(\rF)(x) 
		:= \int_{\bR^{md}} \re^{2\pi i x\cdot (\xi_1+\cdots+\xi_m)}\mathfrak{m}(\xi) \prod_{i=1}^m \widehat{f}_i(\xi_i) ~\mathrm{d}\xi.
	$$
	Suppose that there is a constant $C>0$ such that $|\mathfrak{m}(\xi)|\leq C(1+|\xi|)^{-s}$ for some $s>0$, 
	and $\supp(\mathfrak{m}) \subset \bA^{md}(2^j)$.
	Then we have
	\begin{align}
		\| T_{\mathfrak{m}}(\rF) \|_{L^2(\bR^d)} \leq &C 2^{-j(s - \frac{(m-1)d}{2})} \prod_{i=1}^m \|f_i\|_{L^2(\bR^d)}.
	\end{align}
\end{lemma}

\begin{proof}	
	For $L^2$ norm we begin with
	\begin{align}\label{230904_2006}
		T_\mathfrak{m}(\rF)(x) = \Psi_j \ast T_\mathfrak{m}(\rF)(x, \dots, x),
	\end{align}
	where $\widehat{\Psi}_j$ denotes the $j$-th Littlewood-Paley decomposition in $\bR^{md}$.
	Taking $L^2$ norm to \eqref{230904_2006}, it yields
	\begin{align}\label{230925_1434}
	\begin{split}
		\| T_\mathfrak{m}(\rF) \|_{L^2(\bR^d)}^2 
		&= \int_{\bR^d} \Big| \int_{\bR^{md}} T_\mathfrak{m}(\rF)(y) \Psi_j(x-y_1, \dots, x-y_m)~\mathrm{d}y\Big|^2 ~\mathrm{d}x\\
		&\le \int_{\bR^d}  \int_{\bR^{md}} \Big|T_\mathfrak{m}(\rF)(y) \Big|^2 |\Psi_j(x-y_1, \dots, x-y_m)|~\mathrm{d}y ~\mathrm{d}x\\
		&\le 2^{j(m-1)d} \| T_\mathfrak{m}(\rF) \|_{L^2(\bR^{md})}^2.
	\end{split}
	\end{align}
	Here we apply H\"older's inequality in the first inequality and compute $x$-integral first in the second inequality.
	By the Plancherel theorem and the conditions of $\mathfrak{m}$, we obtain
	\[
		\| T_\mathfrak{m}(\rF) \|_{L^2(\bR^d)} \leq 2^{-j(s - \frac{(m-1)d}{2})} \prod_{i=1}^m \|f_i\|_{L^2(\bR^d)}.
	\]
	By \eqref{230925_1434} one can check that for general $p\geq1$
	\[
		\| T_\mathfrak{m}(\rF) \|_{L^p(\bR^d)} \leq 2^{j\frac{(m-1)d}{p}} \| T_\mathfrak{m}(\rF) \|_{L^p(\bR^{md})}.
	\]
\end{proof}

\section{Proof of Theorem~\ref{thm_loc}}

We make use of the Littlewood-Paley decomposition to obtain
\begin{align}
	\mathrm{M}_\sigma^{loc} (\rF)(x) \leq \sum_{j\geq0} \mathrm{M}_{\sigma_j}^{loc} (\rF)(x).
\end{align}
Note that $ \sigma_j  = \sigma\ast \Psi_j$ for $j>0$ and $\sigma_0 = \sigma\ast\Phi$.
It is worth noting that $\mathrm{M}_{\sigma_0}^{loc}(\rF)$ is bounded pointwisely by product of the Hardy-Littlewood maximal functions $\mathrm{M}_{HL}(f_i)$.
Thus we only have to consider $j\geq1$ cases.

Note that $\mathrm{M}_{\sigma_j}^{loc}(\rF)(x) = \sup_{1<t<2}|\mathrm{A}_{\sigma_j}(\rF)(x,t)|$, 
where $\mathrm{A}_{\sigma_j}(\rF)(x,t)$ is given by
$$
	\mathrm{A}_{\sigma_j}(\rF)(x,t) 
	= \int_{\bR^{md}} \re^{2\pi i x\cdot (\xi_1+\cdots+\xi_m)} \widehat{\sigma}_j(t\xi) \prod_{i=1}^m \widehat{f}_j(\xi_i) ~\mathrm{d}\xi.
$$
Since $1<t<2$, one can say that $2^{j-2} \leq |\xi|\leq 2^{j+2}$ in the support of $\widehat{\sigma}_j(t\cdot)$, hence $|\xi_i|\lesssim 2^{j}$ for each $i=1, \dots, m$. 
Moreover, there is at least one $i=1, \dots, m$ such that $2^{j-2} \leq |\xi_i| \leq 2^{j+2}$.
We first obtain the following lemma:
\begin{lemma}\label{lem_loc_max}
Let $s(m,d) = s  - \frac{(m-1)d}{2}-\frac{1}{2}$. Then we have
	\begin{align}\label{230914_0023}
		\Big\| \sup_{1<t<2}\Big| \mathrm{A}_{\sigma_j}(\rF)\Big| \Big\|_{L^{2/m}(\bR^d)} \leq C 2^{-js(m,d)}\prod_{i=1}^m\|f_i\|_{L^2(\bR^d)}.
	\end{align}
\end{lemma}
As a consequence of the lemma~\ref{lem_loc_max}, we show theorem~\ref{thm_loc}. In fact, by the lemma~\ref{lem_loc_max}, it follows that
\begin{align}\label{230830_2323}
\begin{split}
	\|\mathrm{M}_\sigma^{loc}(\rF)\|_{L^{2/m}(\bR^d)}^{2/m} 
	\leq &\sum_{j\geq0} \Big\| \sup_{1<t<2}\Big| \mathrm{A}_{\sigma_j}(\rF)(\cdot, t)\Big| \Big\|_{L^{2/m}(\bR^d)}^{2/m}\\
	\leq &C \sum_{j\geq0} 2^{-js(m,d)2/m}\prod_{i=1}^m\|f_i\|_{L^2(\bR^d)}^{2/m}.
\end{split}
\end{align}
 Since $s(m,d)>0$, this proves the theorem.
Therefore it remains to prove Lemma~\ref{lem_loc_max}.

\subsection{Proof of Lemma~\ref{lem_loc_max}}

In order to show lemma ~\ref{lem_loc_max}, we use the following lemma in \cite[Lemma~2.3]{KoLeeOh_2022}. 
\begin{lemma}\label{lem_local}
	Let $R = 1+4 \text{diam}(\supp(\sigma))$ and 
	\[
		K_j(\mathbf{x},t) \coloneqq \int_{\bR^{md}} \mathrm{e}^{2\pi i \mathbf{x}\cdot\xi} \int \mathrm{e}^{-2\pi i t(y\cdot\xi)}~\mathrm{d}\sigma(y) \widehat{\Psi}(2^{-j}\xi)~\mathrm{d}\xi
	\]
	for $j\geq0$ and $(\mathbf{x},t) \in \bR^{md} \times \bR$.
	If $|\mathbf{x}| \geq R$ and $|t|\leq 2$, then $|K_j(\mathbf{x},t)|\lesssim 2^{-jN}(1+|\mathbf{x}|)^{-N}$ for any $N>0$.
\end{lemma}

\begin{proof}
	By Fubini's theorem and change of variables, we have
	\begin{align}
	\begin{split}
		K_j(\mathbf{x},t) 
		= 2^{jmd}\int  \int_{\bR^{md}} \mathrm{e}^{2\pi i 2^j (\mathbf{x}\cdot\xi- ty\cdot\xi)}  \widehat{\Psi}(\xi)~\mathrm{d}\xi\mathrm{d}\sigma(y).
	\end{split}
	\end{align}
	Note that $|\nabla_\xi(\mathbf{x}\cdot\xi- ty\cdot\xi)| \geq |\mathbf{x}|/2$ for $|\mathbf{x}|\geq R$ and $|t|\leq2$.
	Then by iterated use of integration by parts in $\xi$, we have
	$$
		|K_j(\mathbf{x},t)|\chi_{|\mathbf{x}|\geq R}(\mathbf{x}) \chi_{|t|\leq2}(t) \lesssim 2^{jmd} (1+2^{j}|\mathbf{x}|)^{-L}\chi_{|\mathbf{x}|\geq R}(\mathbf{x})
	$$
	for any $L>0$.
	Taking $L\geq md+2N$ for a fixed number $N>0$, it follows that 
	$$
		|K_j(\mathbf{x},t)|\chi_{|\mathbf{x}|\geq R}(\mathbf{x}) \chi_{|t|\leq2}(t)  \leq 2^{-jN}(1+|\mathbf{x}|)^{-N}.
	$$
\end{proof}
We split the integral range $\bR^d$ into the Ball of radius $R$ centered at the origin and the complement, and apply Lemma~\ref{lem_local} to get
\begin{align}
\begin{split}
	\Big\| &\sup_{1<t<2}\Big| \mathrm{A}_{\sigma_j}(\rF)\Big| \Big\|_{L^{2/m}(\bR^d)} \\
	\lesssim &\Big\| \sup_{1<t<2}\Big| \mathrm{A}_{\sigma_j}(\rF)\Big| \Big\|_{L^{2/m}(\bB^d(0,R))} + \Big\| \sup_{1<t<2}\Big| \mathrm{A}_{\sigma_j}(\rF)\Big| \Big\|_{L^{2/m}(\bB^d(0,R)^\mathsf{c})}\\
	\lesssim &\Big\| \sup_{1<t<2}\Big| \mathrm{A}_{\sigma_j}(\rF)\Big| \Big\|_{L^{2/m}(\bB^d(0,R))} + 2^{-jN}\Big\| \prod_{i=1}^m M_{HL}f_i\Big\|_{L^{2/m}(\bB^d(0,R)^\mathsf{c})},
\end{split}
\end{align}
where $A^\mathsf{c}$ denotes the complement of a set $A$.
Therefore to prove Lemma~\ref{lem_loc_max}, it suffices to show \eqref{230914_0023} with left-hand side replaced in to $L^{2/m}(\bB^d(0,R))$-norm since the other part can be controlled by the right side norm from the Hardy-Littlewood maximal estimates.

We define $\widetilde{\mathrm{A}}_{\sigma_j}$ as
\begin{align*}
	\widetilde{\mathrm{A}}_{\sigma_j}(\rF)(x,t)
	&= \int_{\bR}\int_{\bR^{md}} \re^{2\pi i (x\cdot(\xi_1+\cdots+\xi_m)+t\tau)} m(\xi, \tau) \widehat{\Psi}(2^{-j}\xi) \widehat{\varphi}(2^{-j}R^{-1}\tau) \prod_{i=1}^m \widehat{f}_i(\xi_i)~\mathrm{d}\xi \mathrm{d}\tau,
\end{align*}	
where $m(\xi, \tau)= \int_\bR \int e^{-2\pi i s(\tau +y\cdot\xi)} \rho(s)~\mathrm{d}\sigma(y) \mathrm{d}s$ with a smooth function $\rho\in C_c^\infty(\bR)$ such that $\rho\equiv1$ on $[1,2]$ and $R = 1+4 \text{diam}(\supp(\sigma))$.
Let $\mathcal{F}$ denote the space-time Fourier transform.
For $\mathrm{R}_{\sigma_j} = \mathrm{A}_{\sigma_j} - \widetilde{\mathrm{A}}_{\sigma_j}$, $\cF[\rR_{\sigma_j}(\rF)]$ is supported outside of $\bB^{d+1}(0, m2^{j+1})$, and the kernel of $\rR_{\sigma_j}$ has nice decay property.
Then, we have the following lemma:
\begin{lemma}\label{lem_out}
	Let $m(\xi, \tau)\widehat{\Psi}_j(\xi)(1 - \widehat{\varphi}_j)(R^{-1}\tau)$ be a symbol of $\mathrm{R}_{\sigma_j}$ with respect to space-time Fourier transform.
	Then for any $N\geq d$ we have
	\begin{align}
		\cF^{-1}[m(\xi, \tau)\widehat{\Psi}_j(\xi)(1 - \widehat{\varphi}_j)(R^{-1}\tau)](\mathbf{x},t) \leq C_N 2^{-jN} (1+|\mathbf{x}|)^{-N} (1+|t|)^{-N},
	\end{align}
	where $C_N$ is independent of $j$ and $\mathbf{x}, \xi \in \bR^{md}$.
\end{lemma}

\begin{proof}
	Note that 
	\[
		m(\xi, \tau) = \int \widehat{\rho}(\tau + y\cdot \xi)~\mathrm{d}\sigma(y).
	\]
	For multi-indeces $\alpha$ with $|\alpha|=N$, it follows that
	\begin{align*}
		| \partial_\xi^\alpha \partial_\tau^N \big(m(\xi, \tau) \widehat{\Psi}_j(\xi)(1 - \widehat{\varphi}_j)(R^{-1}\tau)\big)|
	\end{align*}
	is bounded by finite sum of the following term:
	\begin{align}\label{eq_2308_1}
		\int  | \widehat{\rho}^{(|\alpha_1| +N)}(\tau + y\cdot\xi) | ~\mathrm{d}\sigma(y) \times 2^{-j|\alpha_2|} \widehat{\Psi}_j(\xi) (1 - \widehat{\varphi}_j)(R^{-1}\tau),
	\end{align}
	where $\alpha = \alpha_1 + \alpha_2$.
	Since $|\xi|\sim 2^j$, $|\tau| > 2^j R$, and $\rho \in C_c^\infty(\bR)$, for any $L\geq0$ there exists a constant $C_L$ such that
	\begin{align}
	  \eqref{eq_2308_1} 
	  &\leq C_L  2^{-j|\alpha_2|} \frac{1}{(1+ |\tau|)^L} \widehat{\Psi}_j(\xi)(1 - \widehat{\varphi}_j)(R^{-1}\tau)\nonumber\\
	  &\leq C_L \frac{1}{(1+ |\tau|)^{L/2}} (R 2^j)^{-L/2}\widehat{\Psi}_j(\xi)\label{eq_2308_2}.
	\end{align}
	By \eqref{eq_2308_2} and the integration by parts, it follows that
	\begin{align*}
		\cF^{-1}[m(\xi, \tau)\widehat{\Psi}_j(\xi)&(1 - \widehat{\varphi}_j)(R^{-1}\tau)](\mathbf{x},t)\\
		\lesssim_R \,&C_L (1+|\mathbf{x}|)^{-N} (1+|t|)^{-N} 2^{-j(L/2 -md)}.
	\end{align*}
	If we choose $L\geq2(N+md)$, then we complete the proof of the lemma. 
\end{proof}

Now we return to the proof of Lemma ~\ref{lem_loc_max}. For $\mathrm{R}_{\sigma_j}$, apply Lemma~\ref{lem_out} to have
\[
	\sup_{1<t<2} | \mathrm{R}_{\sigma_j}(\rF)(x,t)| \leq 2^{-jN} \prod_{i=1}^m \mathrm{M}_{HL}(f_i)(x).
\]
$N$ could be sufficiently large, so we obtain the desired result.

To deal with $\widetilde{\mathrm{A}}_{\sigma_j}$, we observe that the Fourier transform of $\mathrm{A}_{\sigma_j}$.
The Fourier transform in $x$ is
\begin{align}
\begin{split}
	\widehat{\mathrm{A}_{\sigma_j}(\mathrm{F})}(\cdot, t)(\xi_1)
	=&\int_{\bR^{(m-1)d}} \widehat{\mathrm{d}\sigma}_j(t\xi_1 - t\xi_2,  \dots, t\xi_{m-1}-t\xi_m, t\xi_m) \\
	&\quad\quad\times \widehat{f}_1(\xi_1-\xi_2) \cdots \widehat{f}_{m-1}(\xi_{m-1}-\xi_m) \widehat{f}_m(\xi_m) 
	~\mathrm{d}\xi_2\cdots \mathrm{d}\xi_m.
\end{split}
\end{align}
Since $\supp(\widehat{f_i}) \subset \mathbb{B}^d(0, 2^j)$ and $1<t<2$,
it implies that $\xi_1 \in \mathbb{B}^d(0, m 2^{j+1})$,  and $\cF[\widetilde{\rA}_{\sigma_j}(\rF)](\xi_1,\tau)$ is supported in $\bB^{d+1}(0, m2^{j+1})$.
Since $|\widehat{\mathrm{d}\sigma}(\xi)| \lesssim (1+|\xi|)^{-s}$ and $\widehat{\mathrm{d}\sigma_j}$ is supported in $\bA^{md}(2^j)$, we use Young's inequality in $t$-variable and Lemma~\ref{lem_decay} in $x$-variable to obtain
\begin{align}\label{230831_2300}
\begin{split}
	\| \widetilde{\rA}_{\sigma_j}(\rF)(\cdot, \cdot)\|_{L^2(\bR^d\times [1,2])} 
	&\le \| {\rA}_{\sigma_j}(\rF)(\cdot, \cdot)\|_{L^2(\bR^d\times [1,2])}\\
	&\le C \,2^{-j(s-\frac{(m-1)d}{2})} \prod_{i=1}^m \|f_i\|_{L^2(\bR^d)}.
\end{split}
\end{align}
To estimate our maximal estimate, we use H\"older's inequality and Lemma~\ref{lem_a_meas},
which is due to the equivalence \eqref{a_dim_est}.
Then, 
\begin{align}\label{231120_1745}
\begin{split}
	\Big\| \sup_{1<t<2} \Big| \widetilde{\rA}_{\sigma_j}(\rF)(\cdot, t) \Big| \Big\|_{L^{2/m}(\bB^d(0,R))} 
	&\le R^{\frac{(m-1)d}{2}} \Big\| \sup_{1<t<2} \Big| \widetilde{\rA}_{\sigma_j}(\rF)(\cdot, t) \Big| \Big\|_{L^{2}(\bB^d(0,R))} \\
	&\lesssim 2^{j/2} \| \widetilde{\rA}_{\sigma_j}(\rF)(\cdot, \cdot)\|_{L^2(\bR^d\times [1,2])}.
\end{split}
\end{align}
Combining two estimates \eqref{230831_2300} and \eqref{231120_1745}, we obtain
\[
	\Big\| \sup_{1<t<2} \Big| \widetilde{\rA}_{\sigma_j}(\rF)(\cdot, t) \Big| \Big\|_{L^{2/m}(\bB^d(0,R))} 
	\lesssim 2^{-j(s-\frac{(m-1)d}{2} - \frac{1}{2})} \prod_{i=1}^m \|f_i\|_{L^2(\bR^d)}.
\]
Hence, the proof is completed.

We end this section with the proof of Corollary~\ref{Prop1}.
\begin{proof}[Proof of Corollary~\ref{Prop1}]
	From the Lemma \ref{lem_loc_max} we have 
	\[
		\Vert \mathrm{M}^{loc}_{\sigma_j}\Vert_{L^2\times\cdots\times L^2\rightarrow L^{2/m}}\lesssim 2^{-j(k/2-\frac{(m-1)d}{2}-\frac{1}{2})}.
	\]
	On the other hand, observe that 
	\[
		\mathrm{M}^{loc}_{\sigma_j}(f_1,\dots, f_m)(x)\lesssim 2^j \prod_{i=1}^m M_{HL}f_i (x).
	\]
	Thus we also have $\Vert \mathrm{M}^{loc}_{\sigma_j}\Vert_{L^{p_1}\times\cdots\times L^{p_m}\rightarrow L^{p}}\lesssim 2^j$ for $1/m<p\leq\infty$, $1<p_1,\dots, p_m\leq\infty$ with $1/p = 1/p_1 + \cdots +1/p_m$.
	Therefore, applying the real interpolation theory of multi(sub)linear operators and summing over $j$ we get $L^{p}(\mathbb{R}^d)\times\cdots\times L^{p}(\mathbb{R}^d)\rightarrow L^{p/m}(\mathbb{R}^d)$ estimate of $\mathrm{M}^{loc}_{\sigma}$ for $p>\frac{k-(m-1)d+1}{k-(m-1)d}$.
\end{proof}

\section{Proof of Theorem~\ref{thm_main}}

We first recall that the global maximal function $\mathrm{M}_\sigma$ is defined by 
\begin{align}
  \mathrm{M}_{\sigma}(\rF)(x)
  =\sup_{k\in\mathbb{Z}} \sup_{1<t<2} \Big|\int_{\Sigma} \, \prod_{i=1} f_{i}(x-2^{-k}ty_i) ~\mathrm{d}\sigma(y)\Big|.
\end{align}
For any fixed $k\in\mathbb{Z}$, we define the Littlewood-Paley decomposition on $\bR^d$ as
 \begin{align}\label{proj_identity}
P_{<k} =  I -\sum^{\infty}_{n=0}P_{k+n}, 
 \end{align}
 where $P_k$ denotes the $k$-th Littlewood-Paley projection and the identity operator $I$.
 Then we may write
 \begin{equation}\label{proj_mlinear_id}
 \begin{aligned}
  \prod_{i=1}^m f_i 
  &= \prod_{i=1}^m \Big(P_{<k}f_i + P_{k\leq}f_i \Big)\\
  &= \Bigl(\prod_{\mu=1}^m P_{<k}f_\mu \Bigr) 
  +\Bigl( \prod_{\nu=1}^m P_{k\leq}f_\nu\Bigr)\\
  &\quad+ \sum_{\alpha= 1}^{m-1} \frac{1}{\alpha! (m-\alpha)!}\sum_{\tau\in S_m} \Bigl(\prod_{\mu=1}^\alpha P_{<k}f_{\tau(\mu)} \Bigr) \Bigl( \prod_{\nu=\alpha+1}^m P_{k\leq}f_{\tau(\nu)}\Bigr).
 \end{aligned}
\end{equation}
For ${\bf{n}}=(n_1,\cdots,n_m) \in \mathbb{N}_0^m = (\mathbb{N} \cup \{0\})^m$, we define for $\alpha=1, \dots, m-1$ and $\tau\in S_m$,
\begin{align*}
	\mathfrak{A}_{k}^{\alpha,\tau}(\rF)(x)
    &:=\sup_{1<t<2}\Big|\int_{\Sigma} \Bigl(\prod_{\mu=1}^\alpha P_{<k}f_{\tau(\mu)}(x-2^{-k} t y_{\tau(\mu)}) \Bigr) \Bigl( \prod_{\nu=\alpha+1}^m f_{\tau(\nu)} (x - 2^{-k} t y_{\tau(\nu)})\Bigr)~~\mathrm{d}\sigma(y)\Big|
    \end{align*}
    and
    \begin{align*}
    	\mathfrak{A}_{k}^{m,\tau}(\rF)(x) =\mathfrak{A}_{k}^m(\rF)(x)
    &:=\sup_{1<t<2}\Big|\int_{\Sigma} \Bigl(\prod_{\mu=1}^m P_{<k}f_\mu(x-2^{-k} t y_\mu) \Bigr)~\mathrm{d}\sigma(y)\Big|.
    \end{align*}
    Let us define
    \begin{align}
	\mathfrak{M}_{\bf{n}}(\rF)
    &:=\sup_{k\in\mathbb{Z}} \sup_{1<t<2}\Big|\int_{\Sigma} \prod_{i=1}^m P_{k+n_i} f_i(x - 2^{-k} t y_i)~~\mathrm{d}
    \sigma(y) \Big|,\label{M_n_F}\\
	\mathfrak{S}_{\bf{n}}^q(\rF)
    &:=\Big\| \sup_{1<t<2} \Big| \int_{\Sigma} \prod_{i=1}^m P_{k+n_i} f_i(x - 2^{-k} t y_i)~~\mathrm{d}
    \sigma(y) \Big| \Big\|_{\ell^q({k\in\mathbb{Z}})}.\label{S_n_F}
\end{align}
Then, the global maximal function $\mathrm{M}_{\sigma}$ is bounded by a constant multiple of 
\begin{align}
  \sum_{\alpha=1}^m \sum_{\tau\in S_m}\sup_{k\in\bZ} |\mathfrak{A}_k^{\alpha,\tau}(\rF)|+ \sum_{\mathbf{n}\in \mathbb{N}_0^m} \mathfrak{M}_{\mathbf{n}}(\rF).
\end{align}
Since the choice of $\tau$ for fixed $\alpha$ is harmless in estimating $\mathfrak{A}_k^{\alpha,\tau}(\rF)$, we may consider $\mathfrak{A}_k^{\alpha}(\rF)$ instead of $\mathfrak{A}_k^{\alpha,\tau}(\rF)$, which is given by
\begin{align}
	\mathfrak{A}_{k}^{\alpha}(\rF)(x)
    &:=\sup_{1<t<2}\Big|\int_{\Sigma} \Bigl(\prod_{\mu=1}^\alpha P_{<k}f_{\mu}(x-2^{-k} t y_{\mu}) \Bigr) \Bigl( \prod_{\nu=\alpha+1}^m f_{\nu} (x - 2^{-k} t y_{\nu})\Bigr)~~\mathrm{d}\sigma(y)\Big|.\label{A_ell_F}
\end{align}

In order to prove Theorem~\ref{thm_main}, we use an induction argument. We start with the following lemma, which will play a crucial role to show the theorem when $m=2$.
\begin{lemma}\label{lem_m2}
  For $m=2$ and $\alpha=1$ we have
  \[
    \mathfrak{A}_k^\alpha(\rF)(x) \leq M_{HL}(f_1)(x) \times M_\sigma^i(f_2)(x),
  \]
  where $M_\sigma^i$ is defined by \[M_\sigma^i(f)(x) = \sup_{t>0}\int_{\Sigma} | f(x-ty_i)|~\mathrm{d}\sigma(y).\]
\end{lemma}
\begin{proof}
	It suffices to show \[\sup_{1<t<2}\sup_{y\in\Sigma}|P_{<k}f(x-2^{-k}ty)| \lesssim M_{HL}(f)(x).\] 
 Indeed,
  \begin{align*}
    P_{<k}f(x-2^{-k}ty)
    &= \int_{\bR^d} f(z) \, 2^{k d}\,\varphi(2^k(x - 2^{-k}ty -z))~~\mathrm{d}z\\
    &= \int_{\bR^d} f(x+ 2^{-k}z)\, \varphi(ty-z)~~\mathrm{d}z.
  \end{align*}
  Since $y\in\Sigma$ and $1<t<2$, it follows that for any $N>0$
  \[
    |P_{<k}f(x-2^{-k}ty)| \lesssim \int_{\bR^d} |f(x + 2^{-k}z)| \frac{C_N}{(1+|z|)^N}~~\mathrm{d}z \leq M_{HL}(f)(x).
  \]
\end{proof}

Note that $M_\sigma^i$ is bounded on $L^2(\bR^d)$  whenever $|\widehat{\mathrm{d}\sigma}(\xi)| \lesssim (1+|\xi|)^{-s}$ for $s>\frac{1}{2}$ (see \cite{LeeSeo_2023, RdF_1986}). In fact, from the assumption $s>\frac{(m-1)d}{2} +\frac{1}{2}$, it follows $s>\frac12$ for $m=2$. 
Then, by Lemma~\ref{lem_m2}, when $m=2$ we prove
\[
	\Big\| \sum_{\alpha=1}^2 \sup_{k\in\bZ}|\mathfrak{A}_k^\alpha (\rF) | \Big\|_{L^1(\bR^d)} \lesssim \|f_1\|_{L^2(\bR^d)} \times \|f_2\|_{L^2(\bR^d)}.
\]

For the summation of $\mathfrak{M}_\mathbf{n}$ over $\mathbf{n} \in \mathbb{N}_0^m$, we use the following lemma.
\begin{lemma}\label{lem_decay1}
  Let $\mathbf{n} \in \mathbb{N}^m$ and $q=2/m$. Then, we have
 \[
    \|\mathfrak{S}_{\mathbf{n}}^q(\rF)\|_{L^{2/m}(\bR^d)} \lesssim 2^{-\delta(\mathbf{n},m,d)} \prod_{i=1}^m \|f_i\|_{L^2(\bR^d)},
 \]
  where $\delta(\mathbf{n},m,d) = m^{-1/2} |\mathbf{n}| s(m,d)$.
\end{lemma}

Note that $\mathfrak{M}_\mathbf{n} \leq \mathfrak{S}^{q}_\mathbf{n}$ due to $\ell^\infty \to \ell^q$ embedding with $q=2/m$.
Hence, it follows that
\begin{align}
  \| \mathfrak{M}_\mathbf{n}(\rF) \|_{L^{2/m}(\bR^d)} \lesssim 2^{-\delta(\mathbf{n},m,d)} \prod_{i=1}\|f_i\|_{L^{2}(\bR^d)}.
\end{align}
Since $2^{-\delta(\mathbf{n},m,d)} $ is summable over $\mathbf{n}\in\mathbb{N}_0^{m}$ for $s(m,d)>0$, this proves the theorem for the case of  $m=2$.

For the induction, we assume that Theorem~\ref{thm_main} holds for $N$-linear operators with $N=2,\cdots,m-1$. 
Note that we already show it holds that when $m=2$.
Under the assumption, we will show the following lemma.
\begin{lemma}\label{lem_induction}
  For $\alpha=1, \dots, m$, we have
  $$
    \mathfrak{A}_k^\alpha(\rF)(x) 
    \lesssim 
    \prod_{\mu=1}^\alpha M_{HL}(f_\mu)(x) \times 
    \sup_{k\in\bZ} \sup_{1<t<2}\int_{\Sigma} \Big| \prod_{\nu=\alpha+1}^m f_\nu(x - 2^{-k} t y_\nu)\Big|~\mathrm{d}\sigma(y).
  $$
  Moreover, if we assume Theorem~\ref{thm_main} holds for $N$-linear operators with $N=2,\cdots,m-1$,
  then we have
  \begin{align}\label{230904_2148}
 	 \Big\|  \sup_{k\in\bZ} \sup_{1<t<2}\int_{\Sigma} \Big|\prod_{\nu=\alpha+1}^m f_\nu(x - 2^{-k} t y_\nu)\Big|~\mathrm{d}\sigma(y)\Big\|_{L^{2/(m-\alpha)}(\bR^d, \mathrm{d}x)} \lesssim \prod_{\nu=\alpha+1}^m \|f_\nu\|_{L^2(\bR^d)}.
  \end{align}
\end{lemma}

\begin{proof}
Note that the first assertions of the lemma follows directly by the proof of Lemma~\ref{lem_m2}.
For the second assertion, observe that the integrand of the left-hand side of \eqref{230904_2148} is an $(m-\alpha)$-sublinear operator, and the symbol has decay $s>\frac{(m-1)d}{2} + \frac{1}{2}$.
For $\alpha=m$ we have nothing to prove, for $\alpha=m-1$ we have a linear maximal average which is surely bounded on $L^2(\bR^d)$, and for $\alpha=m-2$ we have a bi(sub)linear operator which is already proved by previous steps.

Therefore it suffices to consider $1\leq \alpha\leq m-3$.
In case of $1\leq \alpha\leq m-3$, observe that the Fourier decay $s>\frac{(m-1)d}{2} + \frac{1}{2}$ is clearly larger than $\frac{(m-\alpha-1)d}{2}+\frac{1}{2}$ which is the condition of Theorem~\ref{thm_main} for $(m-\alpha)$-sublinear operators.
Since we assume that Theorem~\ref{thm_main} holds for $N=1, \dots, m-1$, it follows that
\[
	\Big\Vert \sup_{k\in\bZ} \sup_{1<t<2}\int_{\Sigma} \Big|\prod_{\nu=\alpha+1}^m f_\nu(x - 2^{-k} t y_\nu)\Big|~\mathrm{d}\sigma(y) \Big\Vert_{L^{2/(m-\alpha)}(\bR^d, \mathrm{d}x)}
	\leq C \prod_{\nu=\alpha+1}^m \Vert f_\nu \Vert_{L^2(\bR^d)}.
\]
\end{proof}

Note that we assume Theorem~\ref{thm_main} holds for $N$-linear operators with $N=2,\cdots,m-1$ and prove the $N=2$ case.
For general $m$, we make use of Lemma~\ref{lem_induction} to obtain
\begin{align}\label{230102_1603}
\begin{split}
	&\|\mathfrak{A}_k^\alpha(\rF)\|_{L^{2/m}(\bR^d)}\\
	\lesssim &\Big\|\prod_{\mu=1}^\alpha M_{HL}(f_\mu)(x) \times \sup_{0<t}\int_{\Sigma} \Big|\prod_{\nu=\alpha+1}^m f_\nu(\cdot - t y_\nu)\Big|~\mathrm{d}\sigma(y) \Big\|_{L^{2/m}(\bR^d)}\\
	\leq &\Big\|\prod_{\mu=1}^\alpha M_{HL}(f_\mu)\Big\|_{L^{2/\alpha}(\bR^d)}
	\times \Big\|  \sup_{0<t}\int_{\Sigma} \Big|\prod_{\nu=\alpha+1}^m f_\nu(\cdot - t y_\nu)\Big|~\mathrm{d}\sigma(y)\Big\|_{L^{2/(m-\alpha)}(\bR^d)}\\
	\lesssim &\prod_{\mu=1}^\alpha \| f_\mu\|_{L^2(\bR^d)}
	\times \prod_{\nu=\alpha+1}^m \| f_\nu \|_{L^2(\bR^d)}.
\end{split}
\end{align}
By \eqref{230102_1603} and Lemma~\ref{lem_decay1}, Theorem~\ref{thm_main} is true under the assumption that $N$ cases hold for $N=2,\dots,m-1$.
This closes the induction. Hence, it remains to verify Lemma~\ref{lem_decay1} to complete the proof.

\subsection{Proof of Lemma~\ref{lem_decay1}}

We make use of the following scaling:
  \begin{align*}
    \Big\| \sup_{1<t<2}\Big| \int_{\Sigma} \prod_{i=1}^m \,&P_{k + n_i}f_i(x - 2^{-k} t y_i)~~\mathrm{d}\sigma(y) \Big|\Big\|_{L^{2/m}(\bR^d)}\\
    &=2^{-k md/2} \,\| \mathrm{M}_\sigma^{loc}(P_{n_1}f_{1,k}, \dots, P_{n_m}f_{m,k}) \|_{L^{2/m}(\bR^d)},
  \end{align*}
  where $f_{i,k}(x) = f_i(x/2^{k})$.
  Then, by Lemma~\ref{lem_loc_max} for $q=2/m$
  \begin{align*}
    \| \mathfrak{S}_{\mathbf{n}}^{q}(\rF)\|_{L^{2/m}(\bR^d)}^{2/m} 
    &= \sum_{k\in\bZ} 2^{-k d}\, \| \mathrm{M}_\sigma^{loc}(P_{n_1}f_{1,k}, \dots, P_{n_m}f_{m,k}) \|_{L^{2/m}(\bR^d)}^{2/m}\\
    &\lesssim \sum_{k\in\bZ} 2^{-k d}\, 2^{-(\max_{1\leq i\leq m}|n_i|) (2s(m,d))/{m}} \prod_{i=1}^m \| P_{n_i}f_{i, k} \|_{L^{2}(\bR^d)}^{2/m}\\
    &\le \sum_{k\in\bZ} 2^{-k d} \,2^{-(2\delta(\mathbf{n},m,d))/{m}} \prod_{i=1}^m 2^{(dk)/{m}}\| P_{n_i+k}f_{i} \|_{L^{2}(\bR^d)}^{2/m}\\
    &= \sum_{k\in\bZ}  2^{-(2\delta(\mathbf{n},m,d))/{m}} \prod_{i=1}^m \| P_{n_i+k}f_{i} \|_{L^{2}(\bR^d)}^{2/m},
  \end{align*}
where $\delta(\mathbf{n},m,d) = m^{-1/2}|\mathbf{n}|s(m,d)$ since $\max_{1\leq i\leq m}|n_i| \geq m^{-1/2}|\mathbf{n}|$.
  We apply H\"older's inequality to the last line in the above inequalities to obtain
  \begin{align*}
    \| \mathfrak{S}_{\mathbf{n}}^q(\rF)\|_{L^{2/m}(\bR^d)} 
    \lesssim 2^{-\delta(\mathbf{n},m,d)} \prod_{i=1}^m \Bigl(\sum_{k\in\bZ}\| P_{n_i+k}f_{j} \|_{L^{2}(\bR^d)}^{2}\Bigr)^{{1}/{2}}.
  \end{align*}
  By the Littlewood-Paley decomposition and Plancherel theorem, one can see that 
  $$
  \Big(\sum_j \big\|P_jf\big\|_p^p\Big)^{1/p} \lesssim \|f\|_p
  $$ for $p\geq2$.
  Combining these two estimates, we have
  \[
    \| \mathfrak{S}_{\mathbf{n}}^q(\rF)\|_{L^{2/m}(\bR^d)} \lesssim 2^{-\delta(\mathbf{n},m,d)} \prod_{i=1}^m \|f_i\|_{L^2(\bR^d)}.
  \]
  This proves the lemma.

\section*{Acknowledgement}
All three authors have been partially supported by NRF grant no. 2022R1A4A1018904 funded by the Korea government(MSIT) .
They are supported individually by NRF no. RS-2023-00239774(C. Cho), no. 2021R1C1C2008252(J. B. Lee), and BK21 Postdoctoral fellowship of Seoul National University(K. Shuin).
The authors are sincerely grateful to Prof. Saurabh Shrivastava, Ankit Bhojak and Surjeet Singh Choudhary for their kind comments pointing out some errors.


\end{document}